\documentclass[a4paper,12pt]{article}
\usepackage{amsfonts}
\usepackage{mathrsfs}
\usepackage{array,amsmath,amssymb,amsthm,latexsym,amsfonts}
\usepackage{pstricks}
\usepackage{enumerate}

\usepackage{epsfig}
\usepackage{epstopdf}

\usepackage{graphicx}
\usepackage{color}
\usepackage{ifpdf}

\usepackage{caption}

\begin{document}

\newtheorem{lem}{Lemma}[section]
\newtheorem{thm}{Theorem}[section]
\newtheorem{cor}{Corollary}[section]
\newtheorem{pro}{Proposition}[section]
\newtheorem{con}{Conjecture}
\newtheorem{rem}{Remark}
\newtheorem{obs}{Observation}
\theoremstyle{plain}
\newcommand{\D}{\displaystyle}
\newcommand{\DF}[2]{\D\frac{#1}{#2}}

\renewcommand{\figurename}{{\bf Fig}}
\captionsetup{labelfont=bf}

\title{\bf Tight Nordhaus-Gaddum-type upper bound for
total-rainbow connection number of graphs\footnote{Supported by NSFC No.11371205 and 11531011.}}

\author{{\small Wenjing Li$^1$, Xueliang Li$^{1,2}$, Colton Magnant$^3$, Jingshu Zhang$^1$
}\\
      {\small $^1$Center for Combinatorics and LPMC}\\
      {\small Nankai University, Tianjin 300071, China}\\
       {\small liwenjing610@mail.nankai.edu.cn; lxl@nankai.edu.cn;
       jszhang@mail.nankai.edu.cn}\\
       {\small $^2$Department of Mathematics}\\
{\small Qinghai Normal University, Xining, Qinghai 810008, China}\\
       {\small $^3$Department of Mathematical Sciences}\\
      {\small Georgia Southern University, Statesboro,
      GA 30460-8093, USA}\\
      {\small cmagnant@georgiasouthern.edu}
       }
\date{}

\maketitle
\begin{abstract}
A graph is said to be \emph{total-colored} if
all the edges and the vertices of the graph are colored.
A total-colored graph is \emph{total-rainbow connected}
if any two vertices of the graph are connected
by a path whose edges and internal vertices have distinct colors.
For a connected graph $G$, the \emph{total-rainbow connection number} of $G$, denoted by $trc(G)$,
is the minimum number of colors required in a total-coloring of $G$
to make $G$ total-rainbow connected.
In this paper, we first characterize the graphs
having large total-rainbow connection numbers.
Based on this, we obtain a Nordhaus-Gaddum-type upper bound
for the total-rainbow connection number.
We prove that if $G$ and $\overline{G}$ are
connected complementary graphs on $n$ vertices,
then $trc(G)+trc(\overline{G})\leq 2n$ when $n\geq 6$
and $trc(G)+trc(\overline{G})\leq 2n+1$ when $n=5$.
Examples are given to show that the upper bounds are sharp for $n\geq 5$. This completely solves a conjecture in [Y. Ma, Total rainbow connection number and complementary graph, Results in Mathematics 70(1-2)(2016), 173-182].
\\[2mm]


\end{abstract}
\section{Introduction}
All graphs considered in this paper are simple, finite, and undirected.
We follow the terminology and notation of Bondy and Murty in~\cite{Bondy}
for those not defined here.

Consider an edge-colored graph where adjacent edges may have the same color. A path of the graph is called \emph{a rainbow path} if no two edges of the path have the same color. The graph is called \emph{rainbow connected} if for any two distinct vertices of the graph,
there is a rainbow path connecting them. For a connected graph $G$, the \emph{rainbow connection number} of $G$,
denoted by $rc(G)$, is defined as the minimum number of colors that are required to make $G$ rainbow connected. The concept of rainbow connection of graphs was proposed by Chartrand et al.~in \cite{Char1},
and has been well-studied since then.
For further details, we refer the reader to the book \cite{Li}.

In 2014, Liu et al.~\cite{Liu} introduced the concept of \emph{total rainbow connection}, a generalization of rainbow connection.
Consider a total-colored graph, i.e, all its edges and vertices are colored. A path in the graph is called a \emph{total-rainbow path} if all the edges and inner vertices of the path are assigned distinct colors.
The total-colored graph is \emph{total-rainbow connected
(with respect to a total-coloring $c$)}
if every pair of distinct vertices in the graph are connected by a total-rainbow path. In this case, the total-coloring $c$ is called a \emph{total-rainbow connection
coloring} (\emph{$TRC$-coloring}, for short). For a connected graph $G$,
the \emph{total-rainbow connection number} of $G$, denoted by $trc(G)$,
is the minimum number of colors that are required to make $G$ total-rainbow connected.
The following observations are immediate.

\begin{pro}\label{pro3} Let $G$ be a connected graph. Then we have

\ \ $(i)$ $trc(G)=1$ if and only if $G$ is complete;

\ $(ii)$ $trc(G)\geq 3$ if $G$ is noncomplete;

$(iii)$ $trc(G)\geq 2diam(G)-1$, where $diam(G)$ is the diameter of $G$.
\end{pro}

A Nordhaus-Gaddum-type result is a (tight) lower or upper bound
on the sum or product of the values
of a graph parameter for a graph and its complement.
The name ``Nordhaus-Gaddum-type'' is given because
Nordhaus and Gaddum~\cite{Nordhaus} first established
the following type of inequalities for chromatic number of graphs in 1956.
They proved that if $G$ and $\overline{G}$ are complementary graphs on $n$ vertices whose chromatic numbers are $\chi(G)$ and $\chi(\overline{G})$, respectively, then $2\sqrt{n}\leq \chi(G)+\chi(\overline{G})\leq n+1$. Since then, many analogous inequalities of other graph parameters have been considered, such as diameter~\cite{Harary}, domination number~\cite{Harary2}, rainbow
connection number~\cite{CLL}, (total) proper connection number~\cite{Huang}~(\cite{W.Li}), and so on~\cite{AAC, AM, AH, LiMao}.
Both Ma~\cite{Ma} and Sun~\cite{Sun} studied the Nordhaus-Gaddum-type  lower bound of the total-rainbow
connection number. In~\cite{Ma}, Ma proposed the following conjecture.
\begin{con}
Let $G$ and $\overline{G}$ be complementary connected graphs with $n$ vertices.
Does there exist two constants $C_1$ and $C_2$ such that
$trc(G)+trc(\overline{G})\leq C_1n+C_2$, where this upper bound is tight.
\end{con}

In this paper, we give a positive solution to this conjecture.
We prove that, for $n\geq 6$, if both $G$ and $\overline{G}$ are connected, then
\begin{equation*}
trc(G)+trc(\overline{G})\leq 2n.
\end{equation*}

Thus, we prove that $C_{1} = 2$ and $C_{2} = 0$ for $n \geq 6$.

This paper is organized as follows: In Section 2, we list some useful known results on total-rainbow connection number and study some graphs with clear structure. In Section 3, we characterize graphs that have large total-rainbow connection numbers. In the last section, Section 4, we use the results of Sections~2 and~3 to give the upper bound of $trc(G)+trc(\overline{G})$ and show its sharpness.

\section{Preliminaries}
We begin with some notation and terminology.
For two graphs $G$ and $G'$, we write $G'\cong G$
if $G'$ is isomorphic to $G$.
Let $G$ be a graph, we use $V(G)$, $E(G)$, $n(G)$ and $\overline{G}$
to denote the vertex set, the edge set, the order and the complement
of $G$, respectively. Let $d_{G}(u, v)$ denote the distance between vertices $u$ and $v$ in $G$.
The eccentricity of a vertex $u$, written as $ecc_G(u)$,
is $max\{d_G(u,v):v\in V(G)\}$.
The diameter of the graph $G$, written as $diam(G)$,
is $max\{ecc_G(u):u\in V(G)\}$.
The radius of the graph $G$, written as $rad(G)$,
is $min\{ecc_G(u):u\in V(G)\}$.
A vertex $u$ is the \emph{center} of the graph $G$ if $ecc_G(u)=rad(G)$.
We use $(x,y)$ to denote a vertex pair $x$ and $y$.
For a subgraph $H$ of $G$, we use $G\setminus E(H)$ for the subgraph
obtained from $G$ by deleting the edge set $E(H)$,
and we use $G-e$ instead of $G\setminus \{e\}$ for convenience;
similarly, we use
$G-v$ for the subgraph obtained from $G$
by deleting the vertex $v$ together with all its incident edges.
An edge $xy$ is called a \emph{pendent edge} if one of its end vertices,
say $x$, has degree one, and $x$ is called a \emph{pendent vertex}.
Let $G$ be a graph and $U$ be a set of vertices of $G$.
The \emph{$k$-step open neighborhood} of $U$ in $G$, denoted by $N_G^k(U)$,
is $\{v\in V(G):d_G(v,U)=k\}$ for each $k$, where $0\leq k\leq diam(G)$
and $d_G(v,U)=min\{d_G(v,u):u\in U\}$. We write $N_G(U)$ for $N_G^1(U)$
and $N_G(u)$ for $N_G(\{u\})$.
For any two subsets $X$ and $Y$ of $V(G)$, let $E_G[X,Y]$ denote the edge set
$\{xy\in E(G):x\in X,y\in Y\}$.
When $X$ and $Y$ are disjoint, we use $G[X,Y]$
to denote the bipartite subgraph of $G$ with two partitions
$X$ and $Y$, whose edge set is $E_G[X,Y]$.
Throughout this paper, let $\ell$ denote the circumference of a graph $G$,
that is, the length of a longest cycle of $G$.
We use $C_n$ and $P_n$ to denote the cycle and the path
on $n$ vertices, respectively.

Let $c$ be a total-coloring of a graph $G$. We use $c(e),c(v)$
to denote the color of an edge $e$ and a vertex $v$, respectively.
For a subgraph $H$ of $G$, let $c(H)$ be the set of colors
of the edges and vertices of $H$.
Let $P_{xy}$ denote the path between $x$ and $y$ in $G$.

We first present several results that will be helpful later.

\begin{pro}\label{pro1}
If $G$ is a nontrivial connected graph and
$H$ is a connected spanning subgraph of $G$, then $trc(G)\leq trc(H)$.
\end{pro}

\begin{pro}\cite{Liu}\label{pro2}
Let $G$ be a connected graph on $n$ vertices, with $n'$ vertices having
degree at least~$2$. Then $trc(G)\leq n+n'-1$,
with equality if and only if $G$ is a tree.
\end{pro}

\begin{lem}[\cite{Sun}]\label{lem2.1}
For a connected graph $G$ with $t$ cut vertices and cut edges,
we have $trc(G)\geq t$.
\end{lem}

\begin{thm}[\cite{Liu}]\label{thm2.1}
For $3\leq n\leq 12$, the values of $trc(C_n)$ are given in the following table.
\begin{table}[ht]\footnotesize
  \centering
  \begin{tabular}{|c|c|c|c|c|c|c|c|c|c|c|}
  \hline
  $n$        &3 &4 &5 &6 &7 &8 &9 &10 &11 &12\\
  \hline
  $trc(C_n)$ &1 &3 &3 &5 &6 &7 &8 &9  &11 &11\\
  \hline
  \end{tabular}
\end{table}

For $n\geq 13$, we have $trc(C_n)=n.$
\end{thm}

We rewrite the theorem in terms of $\ell$. Note that when the graph is simply a cycle $C_{n}$, we have $\ell = n$ in this case.

\begin{thm}\label{thm2.2}
For $n \geq 3$, we have
\begin{displaymath}
trc(C_{n})=\left\{\begin{array}{ll}
 2n-\ell-2   &\mbox{if $\ell=3,5;$}\\
 2n-\ell-1   &\mbox{if $\ell=4,12$ or $6\leq \ell\leq 10$;}\\
 2n-\ell     &\mbox{if $\ell=11$ or $\ell\geq 13$.}
\end{array}
\right.
\end{displaymath}
\end{thm}

To \emph{split} a vertex $v$ is to replace $v$
by two adjacent vertices, $v'$ and $v''$, and to replace
each edge incident to $v$ by an edge incident to
either $v'$ or $v''$ (but not both), the other end of each
edge remaining unchanged. Note that the resulting graph
is not unique in general. But in this paper, we only
split cut vertices.
So we stipulate that the new edges incident with the same component of $G-v$ are
incident with the same new vertex.
To \emph{subdivide} an edge $e$ is to delete $e$,
add a new vertex $x$, and join $x$
to the ends of $e$.

If a connected graph $G'$ is obtained from a connected graph $G$ by adding a vertex, subdividing an edge or splitting a cut vertex, then we can give $G'$ a $TRC$-coloring based on the $TRC$-coloring of the original graph $G$ with the addition
of at most two new colors. So we have the following observation.

\begin{obs}\label{obs1}
Let $G$ be a connected graph,
and $G'$ be a connected graph
obtained from $G$ by repeatedly
adding a vertex, subdividing an edge or splitting a cut vertex.
If $trc(G)=2n(G)+b$, where $b$ is a constant,
then $trc(G')\leq 2n(G')+b$.
Moreover, if $G'$ is a graph obtained from $G$ by
adding a pendent vertex $v$ to a cut vertex of $G$,
then $trc(G')\leq 2n(G')+b-1$.
\end{obs}

Let $G$ be a connected unicyclic graph and
$C_\ell$ be the cycle of $G$ such that $C_\ell=u_1u_2\dots u_\ell u_1$. Let $\mathcal{T}_G=\{T_i:1\leq i\leq \ell\}$,
where $T_i$ denotes the component containing
$u_i$ in the subgraph $G\setminus E(C_\ell)$.
An element of $\mathcal{T}_{G}$ is called \emph{nontrivial} if it contains at least on edge. Clearly, each $T_i$ is a tree rooted at $u_i$
for $1\leq i\leq\ell$.
If a pendent vertex (leaf) $u$ of $G$ belongs to $T_i$,
then we also say $u$ is a pendent vertex of $T_i$,
and the pendent edge incident with it is a pendent edge of $T_i$.
We say that $T_i$ and $T_j$ are \emph{adjacent (nonadjacent)}
if $u_i$ and $u_j$ are adjacent (nonadjacent) in the cycle $C_\ell$.
If there is only one element in $\mathcal{T}_G$, say $T_1$,
and $T_{1}$ is a nontrivial path in which $u_{1}$ is a leaf,
then we use $B_\ell$ to denote such a graph $G$.
Based on Theorem~\ref{thm2.2} and Observation~\ref{obs1},
we easily get the values of $trc(B_\ell)$.

\begin{thm}\label{thm2.3}
Let $B_\ell$ be a connected graph of order $n$, where $\ell\geq 3$.
Then we have
\begin{displaymath}
trc(B_\ell)=\left\{\begin{array}{ll}
 2n-\ell-2   &\mbox{if $\ell=3,5,7,9$, or $\ell\geq 11$ is odd and $|T_1|\geq3$;}\\
 2n-\ell-1   &\mbox{otherwise.}
\end{array}
\right.
\end{displaymath}
\end{thm}

\begin{proof}
First of all, $trc(B_\ell)\geq 2diam(G)-1=2(n-\lceil\frac{\ell}{2}\rceil)-1$.
For the upper bound, by Theorem~\ref{thm2.2} and Observation~\ref{obs1}, we get
$trc(B_\ell)\leq 2n-\ell-2$ if $\ell=3,5$;
$trc(B_\ell)\leq 2n-\ell-1$ if $\ell=4,12$ or $6\leq \ell \leq 10$;
$trc(B_\ell)\leq 2n-\ell$ if $\ell=11$ or $\ell \geq 13$.
So the result holds when $\ell=3,4,5,6,8,10,12$.

For $\ell=7$ or~$9$, we provide a total-coloring of $B_\ell$ as follows:
let $c(u_2)=c(u_{\lfloor\frac{\ell}{3}\rfloor+2})
=c(u_{\lfloor\frac{2\ell}{3}\rfloor+2})=1$
and $c(u_3)=c(u_{\lfloor\frac{\ell}{3}\rfloor+3})=
c(u_{\lfloor\frac{2\ell}{3}\rfloor+3})=2$.
Then color $u_2u_3,u_3u_4,u_4, u_{4}u_{5},$ $u_{5}, \dots,u_\ell,$ $u_\ell u_1,u_1,u_1u_2$
with the colors $3,4,\dots,\ell-1,3,4,\dots,\ell-1$,
omitting $u_2,u_3,$ $u_{\lfloor\frac{\ell}{3}\rfloor+2},
u_{\lfloor\frac{2\ell}{3}\rfloor+2},$ $u_{\lfloor\frac{\ell}{3}\rfloor+3}$
and $u_{\lfloor\frac{2\ell}{3}\rfloor+3}$.
Finally, assign a fresh color to each other edge or inner vertex,
and color the leaf with~$1$.
It is easy to check that this coloring is a $TRC$-coloring of $B_\ell$ with
$2n-\ell-2$ colors and the result holds in this case.

For $\ell\geq 14$ and $\ell$ is even,
we define the following total-coloring for $B_{\ell}$:
for $1\leq i\leq\ell$,
let $c(u_iu_{i+1})=i$ (where $u_{\ell+1}=u_1$), and $c(u_i)=i+\frac{\ell}{2}$(mod $\ell$),
assign a fresh color to each other edge and inner vertex, and finally color the leaf with $1$.
Clearly, this total-coloring with $2n - \ell - 1$ colors makes $B_\ell$ total-rainbow connected
and so the result holds when $\ell\geq 14$ and $\ell$ is even.

For odd $\ell\geq 11$, with $|T_1|=2$, the total-coloring defined above is also well-defined.
Then $trc(B_\ell)\leq 2n-\ell-1$.
On the other hand, take a total-coloring of
$B_\ell$ with fewer than $2n-\ell-1$ colors.
Then there are~$3$ elements of $C_\ell$ with the same color.
From the proof of Theorem~\ref{thm2.1} in~\cite{Liu},
this is impossible.
So $trc(B_\ell)=2n-\ell-1$ in this case.

For odd $\ell\geq 11$ with $|T_1|\geq 3$, let $u_1v_1,v_1v_2\in E(B_\ell)$.
Now we provide a total-coloring of $B_\ell$ as follows:
We assign
$v_1, v_1u_1, u_1, u_1u_2, \dots, u_{\lceil\frac{\ell}{2}\rceil-1}u_{\lceil\frac{\ell}{2}\rceil}, u_{\lceil\frac{\ell}{2}\rceil}, u_{\lceil\frac{\ell}{2}\rceil}u_{\lceil\frac{\ell}{2}\rceil+1},\\ u_{\lceil\frac{\ell}{2}\rceil+1}, \dots, u_\ell, u_\ell u_1$
with the colors $1, 2, 3, 4,\dots, \ell+1, 1, 2, 3, \dots, \ell, \ell+1$.
Assign a fresh color to each other edge or inner vertex,
and color the leaf with $1$.
It is not hard to check that this total-coloring
is a $TRC$-coloring of $B_\ell$ with $2n-\ell-2$ colors,
and then $trc(B_\ell)=2n-\ell-2$ in this case.

Thus, our result holds.
\end{proof}

\begin{figure}[!t]
\centering
\scalebox{0.9}[0.9]{\includegraphics{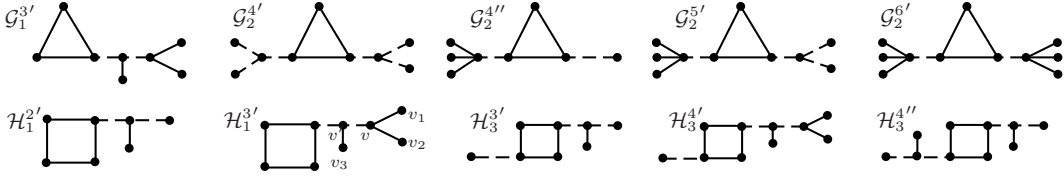}}
\caption{The graph classes ${\mathcal{G}_1^3}',{\mathcal{G}_2^4}',{\mathcal{G}_2^4}'',{\mathcal{G}_2^5}', {\mathcal{G}_2^6}',{\mathcal{H}_1^2}',{\mathcal{H}_1^3}',{\mathcal{H}_3^3}',
{\mathcal{H}_3^4}'$ and ${\mathcal{H}_3^4}''$}
\end{figure}

Throughout this paper, we will define graph classes $\mathcal{G}_i, \mathcal{H}_i, \mathcal{I}_i, \mathcal{J}_i$
to be the sets of graphs with $l=3, 4, 5, 6$, respectively, where $i\geq 0$ is an integer.
For $1\leq i\leq 3$,
let $\mathcal{G}_i=\{G:G$ is a unicyclic graph, $\ell=3$,
$\mathcal{T}_G$ contains only $i$ nontrivial elements$\}$,
and $\mathcal{G}_i^j=\{G\in\mathcal{G}_i:G$ contains $j$ leaves$\}$,
where $j$ is a positive integer.
Let $
{\mathcal{G}_2^4}',
{\mathcal{G}_2^4}'',{\mathcal{G}_2^5}'$ and ${\mathcal{G}_2^6}'$
be the classes of graphs as shown in Fig~$1$ (we use the other graph classes in Fig~$1$ later).
Note that in every graph of Fig~$1$,
each solid line represents an edge,
and each dash line represents a path.
Sun~\cite{Sun} got a sharp upper bound for the total-rainbow connection number
of a connected unicyclic graph with $\ell=3$.
For a connected graph $G$, we use $n'(G)$ to denote the number of
inner vertices of $G$.

\begin{lem}[\cite{Sun}]\label{lem2.2}
For a connected unicyclic graph $G$ with $\ell=3$, we have $trc(G)\leq n(G)+n'(G)-3$;
moreover, $trc(G)= n(G)+n'(G)-3$ if and only if $G\in \mathcal{G}_2^2\cup\mathcal{G}_2^3\cup{\mathcal{G}_2^4}'\cup{\mathcal{G}_2^4}''
\cup{\mathcal{G}_2^5}'\cup{\mathcal{G}_2^6}'\cup\mathcal{G}_3$.
\end{lem}

Let $\mathcal{H}_1=\{G:G$ is a unicyclic graph, $\ell=4,\mathcal{T}_G$
contains only one nontrivial element$\}$ (say $T_1$ is nontrivial),
$\mathcal{H}_2=\{G:G$ is a unicyclic graph, $\ell=4,\mathcal{T}_G$ contains
only two adjacent trivial elements$\}$ (say both $T_3$ and $T_4$ are trivial),
$\mathcal{H}_3=\{G:G$ is a unicyclic graph, $\ell=4,\mathcal{T}_G$ contains
only two nonadjacent trivial elements$\}$ (say both $T_2$ and $T_4$ are trivial),
$\mathcal{H}_4=\{G:G$ is a unicyclic graph, $\ell=4,\mathcal{T}_G$ contains
only one trivial element$\}$ (say $T_4$ is trivial),
$\mathcal{H}_5=\{G:G$ is a unicyclic graph, $\ell=4$,
all elements of $\mathcal{T}_G$ are nontrivial$\}$,
and $\mathcal{H}_i^j=\{G\in\mathcal{H}_i:G$ contains $j$ leaves$\}$,
where $j$ is a positive integer and $1\leq i\leq 5$.
Sun~\cite{Sun} also investigated the total-rainbow connection number
of a connected unicyclic graph with $\ell=4$.

\begin{lem}[\cite{Sun}]\label{lem2.3}
For a connected unicyclic graph $G$ with $\ell=4$,
we have $trc(G)\leq n(G)+n'(G)-3$;
moreover, $trc(G)= n(G)+n'(G)-3$ if and
only if $G\in \mathcal{H}_3^2\cup\mathcal{H}_5^4$.
\end{lem}

\section{Characterizing graphs with large total-rainbow connection number}
To get our main result, we need to characterize
all connected graphs $G$ of order $n$ with $trc(G)=2n-3,2n-4,\dots,2n-8$.

We begin with unicyclic graphs.
For a cycle $C_n$, by Theorem~\ref{thm2.1},
we have $trc(C_n)=2n-5$ if $n=3,4$;
$trc(C_n)=2n-7$ if $n=5,6$;
$trc(C_n)=2n-8$ if $n=7$
and $trc(C_n)\leq 2n-9$ if $n\geq 8$.
In the following, we consider a unicyclic
graph which is not a cycle. Firstly,
we investigate a connected graph with $\ell=3$.
Let ${\mathcal{G}_1^3}'$
be the class of graphs shown in Fig~$1$
and let ${\mathcal{G}_1^3}''=
\mathcal{G}_1^3\setminus {\mathcal{G}_1^3}'$,
${\mathcal{G}_2^4}'''=\mathcal{G}_2^4 \setminus
({\mathcal{G}_2^4}'\cup{\mathcal{G}_2^4}'')$.

\begin{thm}\label{thm3.2}
Let $G$ be a connected unicyclic graph of order $n$, which is not a cycle,
with $\ell=3$. Then
we have $trc(G)= 2n-5$ if $G\in \{B_3\} \cup\mathcal{G}_2^2$;
$trc(G)= 2n-6$ if $G\in \mathcal{G}_1^2\cup\mathcal{G}_2^3
\cup\mathcal{G}_3^3$;
$trc(G)= 2n-7$ if $G\in {\mathcal{G}_1^3}'\cup{\mathcal{G}_2^4}'
\cup {\mathcal{G}_2^4}''\cup\mathcal{G}_3^4$;
$trc(G)=2n-8$ if $G\in {\mathcal{G}_1^3}''\cup{\mathcal{G}_2^4}'''
\cup{\mathcal{G}_2^5}'\cup\mathcal{G}_3^5$;
otherwise, $trc(G)\leq 2n-9$.
\end{thm}

\begin{proof} We need to consider the following three cases.

$(i)$  Consider the graphs in $\mathcal{G}_3$.
It follows from Lemma~\ref{lem2.2} that
$trc(G)=2n-j-3$ if $G\in\mathcal{G}_3^j$, and the results hold.

$(ii)$  Consider the graphs in $\mathcal{G}_2$.
From Lemma~\ref{lem2.2}, we get $trc(G)=2n-5$ if $G\in\mathcal{G}_2^2$, $trc(G)=2n-6$ if $G\in\mathcal{G}_2^3$, $trc(G)=2n-8$ if $G\in {\mathcal{G}_2^5}'$
and $trc(G)\leq2n-9$ if $G\in(\mathcal{G}_2^5\setminus{\mathcal{G}_2^5}')\cup({\bigcup_{j\geq6}\mathcal{G}_2^j})$.
Thus, we only need to consider the graphs $G$ in $\mathcal{G}_2^4$.
Combining Lemmas~\ref{lem2.1} and~\ref{lem2.2}
with the fact that there are $2n-8$ cut edges and cut vertices in $G$,
we have $trc(G)=2n-7$ if $G\in {\mathcal{G}_2^4}'\cup{\mathcal{G}_2^4}''$
and $trc(G)=2n-8$ if $G\in {\mathcal{G}_2^4}'''$.
The results hold in this case.

$(iii)$  Consider the graphs in $\mathcal{G}_1$,
and without loss of generality,
assume $T_1$ is nontrivial.
Note that $trc(G)=2n-5$ if $G\in \mathcal{G}_1^1$,
since $\mathcal{G}_1^1=\{B_3\}$.
Then by Lemma~\ref{lem2.1},
Theorem~\ref{thm2.3}, Observation~\ref{obs1}
and the fact that there a total of $2n-j-5$ cut vertices
and cut edges in $G$,
we get $2n-j-5\leq trc(G)\leq 2n-j-4$ if $G\in \mathcal{G}_1^j$, where $j\geq 2$.

We first focus on the graphs $G\in\mathcal{G}_1^2$.
Without loss of generality, suppose $v_1,v_2$ are the two leaves of $T_1$ and $P_{u_1v}$ is the common part of two paths $P_{u_1v_1}$ and $P_{u_1v_2}$.
Let $c$ be a $TRC$-coloring of $G$ with  $2n-7$ colors.
The $2n-7$ cut vertices and cut edges in $T_1$ cannot be colored the same.
Thus, we have $c(C_3)\subseteq c(T_1)$,
say $c(u_1u_2)\in c(P_{vv_1})$.
If $c(u_1u_3)\in c(P_{u_1v_1})$,
then there is no total-rainbow path between $v_1$ and $u_2$,
deducing $c(u_1u_3)\in c(P_{vv_2})$.
Consider the vertex pair $(v_1,u_2)$, the only
possible total-rainbow path must go through the edge $u_2u_3$,
meaning $c(u_2u_3)\in c(P_{vv_2})$.
Then there is no path to total-rainbow connect
the vertices $v_2$ and $u_3$,
a contradiction. Thus, we have $trc(G)=2n-6$ when $G \in \mathcal{G}_{1}^{2}$.

Next, we concentrate on the graphs $G\in\mathcal{G}_1^3$.
Recall that $trc(G)\geq 2n-8$.
With a similar argument as above, we have $trc(G)=2n-7$
if $G\in {\mathcal{G}_1^3}'$.
It can be easily checked that the total-coloring shown in Fig~$2$
makes $G_1$ total-rainbow connected, so $trc(G_1)=2n(G_1)-8$.
Actually, any graph $G\in {\mathcal{G}_1^3}''$ is obtained from $G_1$
by splitting $u_1$ and repeatedly subdividing the edges of $T_1$.
So by Observation~\ref{obs1},
we obtain $trc(G)\leq2n-8$.
Thus, we have $trc(G)= 2n-8$ for $G\in{\mathcal{G}_1^3}''$.
Similarly, combining Observation~\ref{obs1}
with the total-coloring of $G_2$ shown in Fig~$2$,
we get $trc(G)\leq 2n-9$ for
$G\in \bigcup_{j\geq 4}\mathcal{G}_1^j.$
Now, we get our results in this case.

Thus, our proof is complete.
\end{proof}

Secondly, we consider the connected unicyclic graphs with $\ell=4$.
Let ${\mathcal{H}_1^2}',{\mathcal{H}_1^3}',{\mathcal{H}_3^3}',{\mathcal{H}_3^4}'$
and ${\mathcal{H}_3^4}''$ be the classes of graphs shown in Fig~$1$ and let ${\mathcal{H}_1^2}''=\mathcal{H}_1^2\setminus{\mathcal{H}_1^2}'$
and ${\mathcal{H}_3^3}''=\mathcal{H}_3^3\setminus{\mathcal{H}_3^3}'$.
Let ${\mathcal{H}_4^4}'$ be a subclass of $\mathcal{H}_4^4$, in which
there are two pendent vertices belonging to $T_1$.

\begin{figure}[!t]
\centering
\scalebox{0.9}[0.9]{\includegraphics{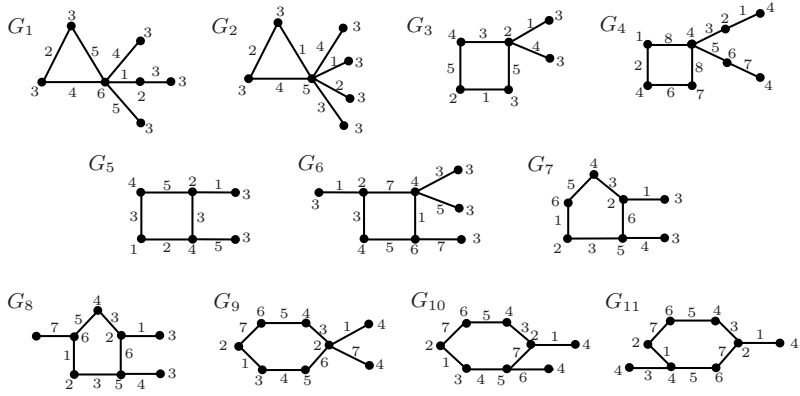}}
\caption{The graphs $G_1$-$G_{11}$}
\end{figure}

\begin{thm}\label{thm3.3}
For a connected unicyclic graph $G$ of order $n$, which is not a cycle, with $\ell=4$,
we have $trc(G)= 2n-5$ if $G\in \{B_4\}\cup\mathcal{H}_3^2$;
$trc(G)=2n-7$ if $G\in {\mathcal{H}_1^2}'\cup \mathcal{H}_2^2\cup
{\mathcal{H}_3^3}'\cup\mathcal{H}_4^3\cup\mathcal{H}_5^4$;
$trc(G)=2n-8$ if $G\in {\mathcal{H}_1^2}''\cup {\mathcal{H}_1^3}'\cup
\mathcal{H}_2^3\cup{\mathcal{H}_3^3}''\cup{\mathcal{H}_3^4}'
\cup{\mathcal{H}_3^4}''\cup{\mathcal{H}_4^4}'$;
otherwise, $trc(G)\leq 2n(G)-9$.
\end{thm}

\begin{proof} In the following argument we distinguish two cases.

($i$) We consider the graphs in $\mathcal{H}_5$.
It follows from Lemma~\ref{lem2.3} that
$trc(G)=2n-7$ if $G\in\mathcal{H}_5^4$,
and $trc(G)\leq 2n-9$ if $G\in \bigcup_{j\geq 5}\mathcal{H}_5^j$.
And the results hold in this case.

($ii$) We consider the graphs in $\mathcal{H}_1$,
and without loss of generality,
we assume $T_1$ is nontrivial.
Note that $trc(G)=2n-5$ if $G\in \mathcal{H}_1^1$,
since $\mathcal{H}_1^1=\{B_4\}$.
With Observation~\ref{obs1} and $trc(G)\geq 2diam(G)-1$,
we also have $trc(G)=2n-5$ if $G\in \mathcal{H}_3^2$.

($a$) We first focus on the graphs in ${\mathcal{H}_1^2}$.
We use $v_1,v_2$ to denote the pendent vertices of $G$,
and $P_{u_1v}$ the common part
of two paths $P_{u_1v_1}$ and $P_{u_1v_2}$.
Let $c$ be a $TRC$-coloring of $G$
with $2n-9$ colors.
Since $T_1$ contains $2n-9$ cut vertices and cut edges of $G$,
these vertices and edges have pairwise distinct colors
and $|c(T_1)|=2n-9$.
Thus, we deduce $c(C_4)\subseteq c(T_1)$.
Without loss of generality, we assume that $c(u_1u_2)\in c(P_{vv_1})$.
Consider the vertex pair $(v_1,u_2)$, the only
possible total-rainbow path must go through the path $u_1u_4u_3u_2$,
implying $\{c(u_2u_3), c(u_3u_4)\}\subseteq P_{vv_2}$.
Then there is no path to total-rainbow connect the vertices
$v_2$ and $u_3$, a contradiction.
The argument is similar if $c(u_1u_2)\in c(P_{u_1v_2})$.
Thus, we conclude that $trc(G)\geq 2n-8$.
Actually, a graph $G$ in $\mathcal{H}_1^2$
is obtained from $G_3$ or $G_4$
by splitting $u_1$ and repeatedly subdividing the edges of $T_1$.
It can be easily verified that the total-colorings shown in Fig~$2$
make $G_3$ and $G_4$ total-rainbow connected, respectively.
Together with Observation~\ref{obs1},
we obtain $trc(G)\leq 2n-7$ if $G\in {\mathcal{H}_1^2}'$
and $trc(G)\leq 2n-8$ if $G\in {\mathcal{H}_1^2}''$.
Now, we get $trc(G)=2n-8$ if $G\in {\mathcal{H}_1^2}''$.
Since $trc(G)\geq 2diam(G)-1$,
we have $trc(G)=2n-7$ if $G\in {\mathcal{H}_1^2}'$.

With an analogous argument as above,
the results hold for the graphs in $\mathcal{H}_2^2\cup\mathcal{H}_2^3
\cup\mathcal{H}_3^3\cup\mathcal{H}_4^3\cup{\mathcal{H}_4^4}'$.

($b$) Next, we concentrate on the graphs $G\in{\mathcal{H}_1^3}'$.
By Observation~\ref{obs1} and subcase $(a)$, we have $trc(G)\leq 2n-8$.
Let $vv_1,vv_2,v'v_3$ be the three pendent edges of $T_1$.
Let $c$ be a  $TRC$-coloring of $G$
with $2n-9$ colors.
Since $T_1$ contains $2n-10$ cut vertices and cut edges of $G$,
these vertices and edges have pairwise distinct colors
and $|c(T_1)|=2n-10$.
Thus, there is only one fresh color $x$ left to color
$C_4$ except the vertex $u_1$; and we have $c(C_4)\cap c(T_1)\neq \emptyset$,
since $diam(C_4)=2$.
Without loss of generality, we assume that $c(u_1u_2)=x$
and $c(u_2)=c(vv_1)$.
Then consider the vertex pair $(v_1,u_3)$, the only
possible total-rainbow path must go through the path $u_1u_4u_3$,
implying $c(u_1u_4u_3)=\{x, c(vv_2), c(v'v_3)\}$.
Now consider the vertex pair $(v_2,u_3)$, forcing $c(u_2u_3)=c(v'v_3)$.
Then there is no path to total-rainbow connect
the vertices $v_3$ and $u_3$, a contradiction.
The arguments are similar for other possible occasions.
So $trc(G)\geq 2n-8$.
Thus, we get $trc(G)=2n-8$.

We can consider the graphs in ${\mathcal{H}_3^4}'\cup{\mathcal{H}_3^4}''$
similarly,
and the results hold.
By the $TRC$-coloring of $G_6$ shown in Fig~$2$,
Observation~\ref{obs1} and the above arguments,
we easily get $trc(G)\leq2n-9$ in all other cases.
Thus, we complete the proof.
\end{proof}

Thirdly, we turn to the connected unicyclic graphs $G$ with $\ell\geq 5$.
From Theorem~\ref{thm2.3} and Observation~\ref{obs1},
we have $trc(G)\leq 2n-9$ if $\ell\geq 7$,
and $trc(G)\leq 2n-7$ if $\ell=5$ or~$6$.
Let $\mathcal{J}_1=\{G:G$ is a unicyclic graph, $\ell=6,\mathcal{T}_G$
contains only two nontrivial elements, $diam(G)=n-3\}$.
Based on Proposition~\ref{pro3}~$(iii)$, Theorem~\ref{thm2.3} and Observation~\ref{obs1}
and the $TRC$-colorings of the graphs $G_9,G_{10}$ and $G_{11}$
shown in Fig~$2$, respectively, we can easily get the following result.

\begin{thm}\label{thm3.4}
Let $G$ be a connected unicyclic graph of order $n$, which is not a cycle,
with $\ell \geq 6$.
Then we have $trc(G)= 2n-7$ if $G\in \{B_6\}\cup\mathcal{J}_1$;
otherwise, $trc(G)\leq 2n-9$.
\end{thm}

The only connected unicyclic graphs $G$ that remain to be considered have $\ell= 5$.
Let $\mathcal{I}_1=\{G:G$ is a unicyclic graph, $\ell=5,\mathcal{T}_G$
contains only one nontrivial element$\}$,
$\mathcal{I}_2=\{G:G$ is a unicyclic graph, $\ell=5,\mathcal{T}_G$ contains
only two nonadjacent nontrivial elements$\}$,
$\mathcal{I}_3=\{G:G$ is a unicyclic graph, $\ell=5,\mathcal{T}_G$ contains
only two adjacent nontrivial elements$\}$,
$\mathcal{I}_4=\{G:G$ is a unicyclic graph, $\ell=5,\mathcal{T}_G$ contains
only two adjacent trivial elements$\}$,
$\mathcal{I}_5=\{G:G$ is a unicyclic graph, $\ell=5$,
$\mathcal{T}_G$ contains only two nonadjacent trivial elements$\}$,
and $\mathcal{I}_i^j=\{G\in\mathcal{I}_i:G$ contains $j$ leaves$\}$,
where $j$ is a positive integer and $1\leq i\leq 5$.

\begin{thm}\label{thm3.1}
Let $G$ be a connected unicyclic graph of order $n$, which is not a cycle,
with $\ell = 5$.
Then we have $trc(G)= 2n-7$ if $G\in \{B_5\}\cup\mathcal{I}_2^2$;
$trc(G)=2n-8$ if $G\in \mathcal{I}_1^2\cup\mathcal{I}_2^3\cup\mathcal{I}_3^2\cup
\mathcal{I}_4^3$;
otherwise, $trc(G)\leq 2n-9$.
\end{thm}
\begin{proof}
By Theorem~\ref{thm2.3}, Observation~\ref{obs1},
the $TRC$-colorings of the graph $G_7$ and $G_8$
shown in Fig~$2$,
and the fact that $trc(G)\geq 2diam(G)-1$,
we easily get that $trc(G)= 2n-7$ if $G\in \{B_5\}\cup\mathcal{I}_2^2$;
otherwise, $trc(G)\leq 2n-8$. Moreover, we have $trc(G)\leq 2n-9$ if $G\in
\mathcal{I}_5\cup
(\bigcup_{j\geq 3}\mathcal{I}_3^j)$.

We consider the graphs $G\in \mathcal{I}_1^2$,
and without loss of generality,
we suppose that $T_1$ is nontrivial with two leaves $v_1,v_2$,
and $P_{u_1v}$ is the common part of two paths $P_{u_1v_1}$
and $P_{u_1v_2}$.
Let $c$ be a  $TRC$-coloring of $G$
with $2n-9$ colors.
Since $T_1$ contains $2n-11$ cut vertices and cut edges of $G$,
these vertices and edges have pairwise distinct colors
and $|c(T_1)|=2n-11$.
Thus, there are only two remaining unused colors $x,y$ left to color
$C_5$ other than the vertex $u_1$; and we have $c(C_5)\cap c(T_1)\neq \emptyset$,
since $diam(C_5)=2$.
Consider the vertex pair $(v_1,u_3)$,
we have $c(u_1u_2u_3)\cap c(P_{vv_2})\neq \emptyset$.
Without loss of generality, we assume that $c(u_1u_2)=x,c(u_2)=y$
and $c(u_2u_3)\in c(P_{vv_2})$.
Then consider the vertex pair $(v_2,u_4)$,
the only possible total-rainbow path must go through the path $u_1u_5u_4$,
implying at least one element of $c(u_1u_5u_4)$ belongs to $c(P_{vv_1})$.
Now consider the vertex pair $(v_1,u_4)$,
the only possible total-rainbow path must go through the edge $u_3u_4$,
forcing $c(u_3u_4)\in c(P_{vv_2})$.
Then there is no path to total-rainbow connect
the vertices $v_2$ and $u_4$, a contradiction.
The arguments are similar for other possible occasions.
So $trc(G)\geq 2n-8$.
Thus, we get $trc(G)=2n-8$.
Analogously, the results hold if $G\in \mathcal{I}_2^3$.

Next, we consider the graphs $G\in \mathcal{I}_3^2$,
say $T_1$ and $T_2$ are nontrivial.
Let $v_i$ be the leaf of $T_i$ for $i=1,2$.
Let $c$ be a $TRC$-coloring of $G$
with $2n-9$ colors.
Since $T_1$ and $T_2$ contain $2n-10$ cut vertices and cut edges of $G$,
these vertices and edges have pairwise distinct colors
and $|c(T_1)\cup c(T_2)|=2n-10$.
Thus, there is only one remaining unused color $x$ left to color
$C_5$ other than the vertices $u_1,u_2$; and we have $c(C_5)\cap
(c(T_1)\cup c(T_2))\neq \emptyset$,
since $diam(C_5)=2$.
Consider the vertex pair $(v_1,v_2)$,
the only possible total-rainbow path must go through the edge $u_1u_2$,
forcing $c(u_1u_2)=x$.
Then consider the vertex pair $(v_2,u_5)$,
if the total-rainbow path goes through the path $u_2u_3u_4u_5$,
then at least~$4$ elements of $c(u_2u_3u_4u_5)$
belong to $c(P_{u_1v_1})$,
inducing there is no path to total-rainbow connect the vertex pair $(v_1,u_4)$
or $(v_1,u_3)$. This means that $c(u_1u_5)\in c(P_{u_1v_1})$.
Now, consider the vertex pair $(v_1,u_5)$,
the only possible total-rainbow path must go through the path
$u_3u_4u_5$, implying $c(u_3u_4u_5)\subseteq c(P_{u_2v_2})$.
Then there is no path to total-rainbow connect
the vertices $v_2$ and $u_4$, a contradiction.
So $trc(G)\geq 2n-8$, and $trc(G)=2n-8$ holds.

Finally, we consider the graphs $G\in \mathcal{I}_4^3$,
with $T_1,T_2,$ and $T_3$ all nontrivial.
Let $v_i$ be the leaf of $T_i$ for $i=1,2,3$.
Let $c$ be a $TRC$-coloring of $G$
with $2n-9$ colors.
Since $T_1,T_2$ and $T_3$ contain $2n-10$ cut vertices and cut edges of $G$,
these vertices and edges have pairwise distinct colors
and $|c(T_1)\cup c(T_2)\cup c(T_3)|=2n-10$.
Thus, there is only one remaining unused color $x$ left to color
$C_5$ except the vertices $u_1,u_2,u_3$; and we have $c(C_5)\cap
(c(T_1)\cup c(T_2)\cup c(T_3))\neq \emptyset$,
since $diam(C_5)=2$.
Consider the vertex pair $(v_1,v_3)$,
if the total-rainbow path goes through the path $u_1u_5u_4u_3$,
then at least~$4$ elements of $c(u_1u_3u_4u_5)$
belong to $c(P_{u_2v_2})$,
inducing there is no path to total-rainbow
connect the vertex pair $(v_2,u_4)$
or $(v_2,u_5)$.
So the total-rainbow path must go through
the path $u_1u_2u_3$, implying at least one of $c(u_1u_2)$ and $c(u_2u_3)$
belongs to $c(P_{u_2v_2})$, say
$c(u_1u_2)\in c(P_{u_2v_2})$.
Then consider the vertex pair $(v_1,v_2)$,
the only possible total rainbow path must go through the path $u_1u_5u_4u_3u_2$.
Similarly, there is no total-rainbow path to connected the vertex pair $(v_3,u_4)$
or $(v_3,u_5)$, a contradiction.
The arguments are analogous for the other subcases.
So $trc(G)\geq 2n-8$.
Therefore, we get $trc(G)=2n-8$.

Thus, our proof is complete.
\end{proof}

\begin{figure}[!t]
\centering
\scalebox{1}[1]{\includegraphics{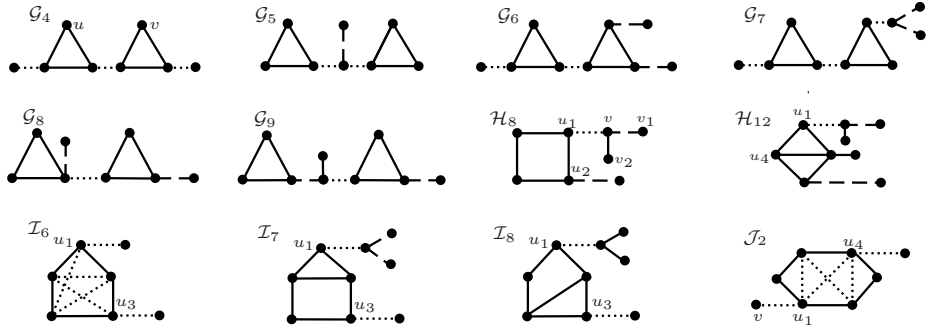}}
\caption{The classes of graphs $\mathcal{G}_4$-$\mathcal{G}_9,
\mathcal{H}_8,\mathcal{H}_{12},\mathcal{I}_6$-$\mathcal{I}_8$, and $\mathcal{J}_2$.}
\end{figure}

So far, we have investigated all the connected unicycle graphs. Next, we focus on the connected graphs with at least two cycles.
Let $\mathfrak{B}=\{G:G$ is a connected graph of order $n$
with at least two cycles$\}$.
When dealing with a graph $G\in \mathfrak{B}$,
we first find a spanning unicyclic subgraph of $G$
and use Proposition~\ref{pro1}
or define a $TRC$-coloring of $G$ to get an upper bound of $trc(G)$,
then we verify whether the bound is tight.
The arguments are similar as above,
so we omit the details.
Let $\mathcal{J}_2=\{G\in \mathfrak{B}:\ell=6,diam(G)=n-3\}$.
The possible structure of each graph in $\mathcal{J}_2$
is shown in Fig~$3$,
where at least one of the edges $u_2u_5,u_2u_6,u_3u_5$
and $u_5u_6$ must exist.
Note that in every graph of Fig~$3$, each solid line represents an edge,
each dash line represents a nontrivial path,
and each dot line represents an edge or a path which can be trivial.
Based on Proposition~\ref{pro1} and Theorem~\ref{thm3.4},
Theorem~\ref{thm3.5} follows.

\begin{thm}\label{thm3.5}
Let $G\in\mathfrak{B}$ be a connected graph of order $n$,
with $\ell \geq 6$.
Then we have $trc(G)= 2n-7$ if $G\in \mathcal{J}_2$;
otherwise, $trc(G)\leq 2n-9$.
\end{thm}

Then consider the connected graphs in $\mathfrak{B}$
with $\ell=5$. Let $\mathcal{I}_6=\{G\in \mathfrak{B}:\ell=5,diam(G)=n-3\}$.
The structure of each graph in $\mathcal{I}_6$
is shown in Fig~$3$,
where at least one of the edges $u_2u_4,u_2u_5$,and $u_3u_5$ must exist.
Let $\mathcal{I}_7,\mathcal{I}_8$ be the classes of graphs shown in Fig~$3$,
respectively,
where $T_1$ is nontrivial.
Let $\mathcal{I}_9$ be a class of graphs in which each graph is obtained
by adding the edge $u_2u_5$ to a graph in $\mathcal{I}_8$.
We set $\mathcal{I}_0=\mathcal{I}_7\cup\mathcal{I}_8\cup
\mathcal{I}_9$.
Let $H_1$ be the graph obtained from $C_5$
by adding a vertex adjacent to
both $u_1$ and $u_3$,
$H_2$ be the graph obtained from $H_1$ by adding the edge $u_1u_3$,
and $H_3$ be the graph obtained from $C_5$ by first adding the edges $u_1u_3$
and $u_1u_4$, and then adding a pendent vertex to the vertex $u_1$.
It is easy to check $trc(H_i)=4=2n-8$ for $i=1,2,3$.
Using Proposition~\ref{pro1}
and Theorem~\ref{thm3.1}, we can easily get the results about
the graphs in $\mathfrak{B}$ with $\ell=5$.

\begin{thm}\label{thm3.6}
Let $G\in\mathfrak{B}$ be a connected graph of order $n$,
with $\ell = 5$.
Then we have $trc(G)= 2n-7$ if $G\in \mathcal{I}_6$;
$trc(G)=2n-8$ if $G\in \{H_1,H_2,H_3\}\cup
\mathcal{I}_0$;
otherwise, $trc(G)\leq 2n-9$.
\end{thm}

Consider the connected graphs in $\mathfrak{B}$ with $\ell=4$.
Let $\mathcal{H}_6=\{G\in\mathfrak{B}:\ell=4,diam(G)=n-2\}$.
Actually, each graph in $\mathcal{H}_6$ is obtained
by adding the edge $u_2u_4$ to a graph in $\{C_4,B_4\}\cup\mathcal{H}_3^2$.
By Observation~\ref{obs1} and Theorem~\ref{thm2.3},
we have $trc(G)=2n-5$ if $G\in \mathcal{H}_6$.
Let $\mathcal{H}_7=\{G\in\mathfrak{B}:\ell=4,diam(G)=n-3\}$.
All the possible structures are illustrated in Fig~$4$.
Note that in each graph of Fig~$4$, each solid line represents an edge,
each dash line represents a nontrivial path,
and each dot line represents an edge or a path which can be trivial.
Let $H_4$ be a graph of order~$6$ with vertex set $V(H_4)=\{v_i:1\leq i\leq6\}$
and edge set $E(H_4)=\{v_1v_i:2\leq i \leq 6\}\cup\{v_2v_i:3\leq i\leq5\}$.
It is not hard to check $trc(H_4)=4=2n-8$.
For convenience we use $\mathcal{H}_8$ to denote a subclass of $\mathcal{H}_2^3$, shown in Fig~$3$.
And let ${\mathcal{H}_8}'=\{G\in\mathcal{H}_8:P_{vv_1}=vv_1\}$.
Let $\mathcal{H}_9$
be a class of graphs, in which each graph is obtained by adding the edge $u_1u_3$
to a graph of ${\mathcal{H}_8}'\cup\mathcal{H}_1^2\cup\mathcal{H}_3^3\cup
\mathcal{H}_4^3\cup\mathcal{H}_5^4$.
And let $\mathcal{H}_{10},\mathcal{H}_{11}$
be classes of graphs, in which each graph is obtained by adding the edge $u_2u_4$
to a graph in $\mathcal{H}_4^3,\mathcal{H}_8\cup{\mathcal{H}_1^2}''\cup
{\mathcal{H}_1^3}'\cup{\mathcal{H}_3^4}'\cup{\mathcal{H}_3^4}''
$, respectively.
Moreover, set $\mathcal{H}'_{10}=\mathcal{H}_{10}
\setminus\mathcal{H}_7^7$.
Let $\mathcal{H}_{12}$ be a class of graphs shown
in Fig~$3$.
Then let $\mathcal{H}_{13}$
be a class of graphs, in which each graph is obtained by
adding the edges $u_1u_3$
and $u_2u_4$
to a graph of ${\mathcal{H}_8}'\cup\mathcal{H}_1^2\cup
\mathcal{H}_4^3\cup\mathcal{H}_5^4$.
We set $\mathcal{H}_0=\mathcal{H}_9\cup\mathcal{H}'_{10}
\cup\mathcal{H}_{11}\cup\mathcal{H}_{12}\cup\mathcal{H}_{13}$.
By Proposition~\ref{pro1} and Theorem~\ref{thm3.3},
it is not hard to obtain the following theorem.

\begin{figure}[!t]
\centering
\scalebox{0.9}[0.9]{\includegraphics{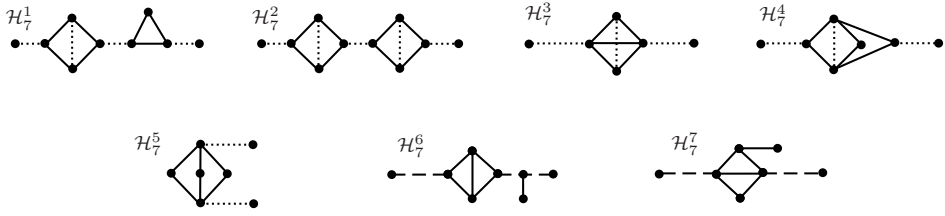}}
\caption{All the possible structures of graphs in $\mathcal{H}_7$.}
\end{figure}

\begin{thm}\label{thm3.7}
For a connected graph $G\in\mathfrak{B}$ of order $n$, with $\ell=4$,
we have $trc(G)= 2n-5$ if $G\in\mathcal{H}_6$;
$trc(G)=2n-7$ if $G\in \mathcal{H}_7$;
$trc(G)=2n-8$ if $G\in \{H_4\}\cup\mathcal{H}_0$;
otherwise, $trc(G)\leq 2n(G)-9$.
\end{thm}

Finally, there remain only the connected graphs in $\mathfrak{B}$
with $\ell=3$ left to consider.
Let $\mathcal{G}_4=\{G\in\mathfrak{B}:\ell=3,diam(G)=n-3\}$.
The possible structure is illustrated in Fig~$3$.
Let $\mathcal{G}_5$-$\mathcal{G}_9$
be the classes of graphs shown in Fig~$3$.
We set $\mathcal{G}_0=\bigcup_{i=5}^9\mathcal{G}_i$.
With the use of Proposition~\ref{pro1} and Theorem~\ref{thm3.2},
we obtain the result concerning the graphs in $\mathfrak{B}$ with
$\ell=3$.

\begin{thm}\label{thm3.8}
For a connected graph $G\in\mathfrak{B}$ of order $n$, with $\ell=3$,
we have $trc(G)=2n-7$ if $G\in \mathcal{G}_4$;
$trc(G)=2n-8$ if $G\in \mathcal{G}_0$;
otherwise, $trc(G)\leq 2n(G)-9$.
\end{thm}

Let $\mathcal{T}^j=\{G:G$ is a tree of order $n$ with $j$ leaves$\}$.
Actually, $\mathcal{T}^2=\{P_n\}$.
Using Proposition~\ref{pro2}, we get $trc(G)=2n-j-1$
if $G\in \mathcal{T}^j$.
We close this section with our final
characterization of graphs having large
total-rainbow connection number, a compilation of the results presented in this section.

\begin{thm}\label{thm3.9}
Let $G$ be a connected graph of order $n$. Then

\ \ $(i)$ $trc(G)= 2n-3$ if and only if $G\cong P_n$;

\ $(ii)$ $trc(G)=2n-4$ if and only if $G\in \mathcal{T}^3$;

$(iii)$ $trc(G)=2n-5$ if and only if $G\in \mathcal{T}^4\cup\{C_3,B_3,C_4,B_4\}
\cup\mathcal{G}_2^2\cup\mathcal{H}_3^2\cup\mathcal{H}_6$;

$(iv)$ $trc(G)=2n-6$ if and only if $G\in \mathcal{T}^5\cup
\mathcal{G}_1^2\cup\mathcal{G}_2^3\cup\mathcal{G}_3^3$;

\ $(v)$ $trc(G)=2n-7$ if and only if $G\in \mathcal{T}^6\cup\{C_5,B_5,C_6,B_6\}\cup{\mathcal{G}_1^3}'\cup{\mathcal{G}_2^4}'
\cup {\mathcal{G}_2^4}''\cup\mathcal{G}_3^4\cup\mathcal{G}_4\cup
{\mathcal{H}_1^2}'\cup \mathcal{H}_2^2\cup
{\mathcal{H}_3^3}'\cup\mathcal{H}_4^3\cup\mathcal{H}_5^4\cup\mathcal{H}_7
\cup\mathcal{I}_2^2\cup\mathcal{I}_6\cup\mathcal{J}_1\cup\mathcal{J}_2$;

$(vi)$ $trc(G)=2n-8$ if and only if $G\in
\mathcal{T}^7\cup\{C_7,H_1,H_2,H_3,H_4\}\cup
{\mathcal{G}_1^3}''\cup{\mathcal{G}_2^4}'''
\cup{\mathcal{G}_2^5}'\cup\mathcal{G}_3^5\cup\mathcal{G}_0\cup{\mathcal{H}_1^2}''
\cup {\mathcal{H}_1^3}'\cup
\mathcal{H}_2^3\cup{\mathcal{H}_3^3}''\cup{\mathcal{H}_3^4}'
\cup{\mathcal{H}_3^4}''\cup{\mathcal{H}_4^4}'\cup\mathcal{H}_0\cup
\mathcal{I}_1^2\cup\mathcal{I}_2^3\cup\mathcal{I}_3^2\cup
\mathcal{I}_4^3\cup\mathcal{I}_0$.
\end{thm}

\section{Upper bound on $trc(G)+trc(\overline{G})$}
Based on the results of Section~$3$,
we give the Nordhaus-Gaddum-type upper bound
of total-rainbow connection number of graphs. At first,
we investigate total-rainbow connection numbers of bridgeless
graphs with diameter~$2$.

\begin{pro}[\cite{H.Li}]\label{pro4}
If $G$ is a bridgeless graph with diameter~$2$,
then either $G$ is~$2$-connected, or $G$ has only one cut-vertex $v$.
Furthermore, the vertex $v$ is the center of $G$, and $G$ has radius~$1$.
\end{pro}

\begin{lem}[\cite{H.Li}]\label{lemD1}
Let $G$ be a bridgeless graph with diameter $2$.
If $G$ has a cut vertex, then $rc(G)\leq3$.
\end{lem}

Given the edge-coloring in the proof of Lemma~\ref{lemD1} of~\cite{H.Li},
we can obtain a $TRC$-coloring of a bridgeless graph $G$ of diameter $2$  with a cut-vertex $v$
by assigning~$4$ to the vertex $v$ and~$5$ to the other vertices of $G$.

\begin{cor}\label{cor4.1}
Let $G$ be a bridgeless graph with diameter $2$.
If $G$ has a cut-vertex, then $trc(G)\leq5$.
\end{cor}

\begin{lem}[\cite{H.Li}]\label{lemD2}
If $G$ is a~$2$-connected graph with diameter $2$, then $rc(G)\leq5$.
\end{lem}

We deal with a~$2$-connected graph in a similar way.

\begin{cor}\label{cor4.2}
If $G$ is a~$2$-connected graph of order $n$ with diameter $2$,
then $trc(G)\leq n-1$.
\end{cor}
\begin{proof} Pick a vertex $v$ in $V(G)$ arbitrarily.
Let $B=\{u\in N_G^2(v):$ there exists a vertex $w$ in $N_G^2(v)$ such that
$uw\in E(G)\}$. We first assume that $B=\emptyset$. We construct a new graph $H$.
The vertex set of $H$ is $N_G(v)$, and the edge set is $\{xy:x,y\in N_G(v)$, $x$
and $y$ are connected by a path $P$ of length at most~$2$ in $G-v$, and
$V(P)\cap N_G(v)=\{x,y\}\}$.
It has been proved in~\cite{H.Li} that the graph $H$ is connected.
Let $T$ be a spanning tree of $H$, and let $(X,Y)$ be the bipartition defined by $T$.
Now, divide $N_G^2(v)$ as follows.
For $N_G^2(v)$, let $A=\{u\in N_G^2(v):u\in N_G(X)\cap N_G(Y)\}$.
For $N_G^2(v)\setminus A$, let $D_1=\{u\in N_G^2(v):u\in N_G(X)\setminus N_G(Y)\}$,
and $D_2=\{u\in N_G^2(v):u\in N_G(Y)\setminus N_G(X)\}$.
As in~\cite{H.Li}, at least one of $D_1$ and $D_2$ is empty, and suppose $D_2=\emptyset$.
Obviously, both $A$ and $D_1$ are independent sets.

Let $X=\{x_1,x_2,\dots, x_s\}$, and let $X'=\{x_1,x_2,\dots, x_t\}$ be a subset
of $X$ such that
$N_G(X')\cap D_1=D_1$ and $t$ is minimum. Moreover,
we choose $x_1$
as a vertex in $X'$ such that $N_G(x_1)\cap D_1$ is maximum, $x_2$
as a vertex in $X'\setminus\{x_1\}$
such that $N_G(x_1)\cap (D_1\setminus N_G(x_1))$ is maximum,
and so on. Clearly,
$t\leq s-1$ since $G$ is~$2$-connected,
and any two vertices of $D_1$ has a common neighbor in $X$
since $diam(G)=2$.
Now, we provide a total-coloring of $G$.
Firstly, we use colors $1,2,3,$ and $4$ to color the edges
$e\in E(G)\setminus E_G[D_1,X]$:
we set $c(e)=1$ if $e\in E_G[v,X]$; $c(e)=2$ if $e\in E_G[v,Y]$;
$c(e)=3$ if $e\in E_G[X,Y]\cup E_G[Y,A]$;
$c(e)=4$ if $e\in E_G[X,A]$, or otherwise.
For the edges $e\in E_G[D_1,X]$ and $1\leq i\leq t$, we set $c(x_id)=5$ where
$d\in D_1\cap(N_G(x_i)\setminus (\bigcup_{j=1}^{i-1}N_G(x_j)))$,
and color the other edges with~$4$.
Then we use color $6$ to color the vertex $v$,~$7$
to color the vertices in $Y$,
and~$8$ to color the vertices in $A\cup D_1$.
We color the vertices $x_1,x_2,\dots,x_t$ with colors $9,10,\dots,t+8$, respectively.
If $|D_1\cap(N_G(x_t)\setminus (\bigcup_{j=1}^{t-1}N_G(x_j)))|=1$,
then we use $t+8$ to color the other vertices in $X_1$;
otherwise, we use $t+9$ to color the other vertices in $X_1$.
Using Table~$2$ in~\cite{H.Li}, one can verify that the above total-coloring
makes $G$ total-rainbow connected.

Clearly, $|D_1|\geq t$. According to the total-coloring
of $G$ defined above, the following is obvious.
When $D_1=\emptyset$, we have $trc(G)\leq 8\leq n-1$ if $n\geq 9$.
When $D_1\neq \emptyset$, the result holds if $|V(G)\setminus D_1|>8$
or $|V(G)\setminus D_1|\leq 8$ and $|D_1|\geq 7$.
Based on Observation~\ref{obs1}, Theorem~\ref{thm2.1} and~\ref{thm2.3},
we can easily check that $trc(G)\leq n-1$
when $4\leq n\leq 14$.

Analogously, we deal with the case $B\neq\emptyset$ and get the result.
For details, we refer to~\cite{H.Li}.
\end{proof}

We have checked all the~$2$-connected graphs $G$ with $n \leq 14$ vertices
and found that $trc(G)\leq n$. So we propose
the following conjecture.
\begin{con}
Let $G$ be a~$2$-connected graph of order $n \ (n\geq 3)$.
Then we have $trc(G)\leq n-1$ if $n\leq 10$ or $n=12$;
and $trc(G)\leq n$, otherwise.
Moreover, the upper bound is tight, which is achieved by the cycle $C_n$ for $n \geq 6$.
\end{con}

Combining Proposition~\ref{pro4} with Corollaries~\ref{cor4.1} and~\ref{cor4.2},
we get the following theorem.
\begin{thm}\label{thm4.1}
If $G$ is a bridgeless graph of order $n$
with diameter~$2$, then $trc(G)\leq n-1$.
\end{thm}

Before stating our main result,
we give some important results needed in the proof of our main theorem.
We first list a theorem concerning the total-rainbow
connection number of a complete bipartite graph $K_{m,n}$, where $n\geq m\geq 2$.

\begin{thm}[\cite{Liu}]\label{thmH}
For $2\leq m\leq n$, we have $trc(K_{m,n})=min(\lceil\sqrt[m]{n}\rceil+1,7)$.

\noindent {\bf Remark:} For $2\leq m\leq n$, we define a \emph{strong $TRC$-coloring} of
a complete bipartite graph $K_{m,n}$ as follows: given a $TRC$-coloring of $K_{m,n}$,
we require that the colors of the vertices in different partitions
are distinct.
In details, when $n>6^m$,
we reserve the total-coloring of $K_{m,n}$;
otherwise, reverse the edge-coloring of $K_{m,n}$,
and assign a new color $p$
to the vertices of one partition
and $q$ to the vertices of the other partition, respectively,
where $p\neq q\in \{1,2,\dots,\lceil\sqrt[m]{n}\rceil+2\}
\setminus \bigcup_{e\in E(K_{m,n})} c(e)$.
According to the proof of Theorem~\ref{thmH},
we still use at most~$7$ colors to get a strong $TRC$-coloring of $K_{m,n}$.

\end{thm}
\begin{thm}[\cite{Zhang}]\label{thmZ}
Let $G$ be a connected graph with connected complement $\overline{G}$. Then

\ $(i)$ if $diam(G) > 3$, then $diam(\overline{G}) = 2$,

$(ii)$ if $diam(G) = 3$, then $\overline{G}$ has a spanning subgraph which is a double star.
\end{thm}

When investigating the total-rainbow connection number of a connected graph $G$ with
diameter~$2$ in terms of its complement $\overline{G}$,
we can give a constant as its upper bound.

\begin{thm}\label{thmS}
Let $\overline{G}$ be a connected graph with $diam(\overline{G})>3$.
Then $trc(G)\leq 7$.
\end{thm}
\begin{proof}
First of all, we note that $G$ must be connected, since otherwise,
$diam(\overline{G})\leq 2$, contradicting the assumption that $diam(\overline{G})\geq 4$.
Choose a vertex $v$ with $ecc_{\overline{G}}(v)=diam(\overline{G})$.
Relabel $N_i(v)=N_{\overline{G}}^i(v)$ for each $1\leq i\leq 3$
and $N_4(v)=\bigcup_{j\geq 4}N_{\overline{G}}^j(v)$.
In the following, we use $N_i$ instead of $N_i(v)$ for convenience.
By the definition of $N_i$,
we know that $G[N_1,N_3]$ (and similarly $G[N_1,N_4]$,$G[N_2,N_4]$)
is a complete bipartite graph.
We give $G$ a total-coloring as follows: we first
set $c(e)=1$ for each edge $e\in E_G[N_1,N_4]\cup E_G[v,N_2]$;
$c(w)=2$ for each vertex $w\in N_2\cup N_4$;
$c(e)=3$ for each edge $e\in E_G[v,N_4]$;
$c(v)=4$;
$c(e)=5$ for each edge $e\in E_G[v,N_3]$,
$c(w)=6$ for each vertex $w\in N_1\cup N_3$;
$c(e)=7$ for each edge $e\in E_G[N_1,N_3]\cup E_G[N_2,N_4]$.
Then color the other edges arbitrarily
(e.g., all of them are colored with~$1$).
One can easily verify that this is a $TRC$-coloring of $G$,
implying $trc(G)\leq 7$.
\end{proof}

By Lemmas~\ref{lem2.2} and~\ref{lem2.3}, we can easily give an upper
bound of $trc(\overline{G})$ when $\overline{G}$ is a connected graph,
whose complement $G$ is a connected graph
with $diam(G)=3$.

\begin{lem}\label{lem4.3}
Let $G$ be a connected graph of order $n$ with diameter $3$.
If $\overline{G}$ is connected, then $trc(\overline{G})\leq n+1$.
Moreover, the equality holds if and only if $\overline{G}$ is
isomorphic to a double star.
\end{lem}

\begin{proof} It follows from Theorem~\ref{thmZ} that $\overline{G}$
contains a double star $T(a,b)$. If $\overline{G}\cong T(a,b)$,
then $trc(G)=n+1$ by Proposition~\ref{pro2}.
Otherwise, $\overline{G}$ contains a unicyclic spanning subgraph
with $\ell=3$ or $\ell=4$ and all the other vertices are leaves.
Therefore, by Proposition~\ref{pro1}, Lemmas~\ref{lem2.2} and~\ref{lem2.3},
we get $trc(\overline{G})\leq n$.
\end{proof}

We know that if $G$ and $\overline{G}$ are connected complementary graphs
on $n$ vertices, then $n$ is at least~$4$.
By Theorem~\ref{thm3.9}, we get that $trc(G)\leq 2n-3$.
Similarly, we have $trc(\overline{G})\leq 2n-3$.
Hence, we obtain that $trc(G)+trc(\overline{G})\leq 4n-6$.
For $n=4$, it is obvious that $trc(G)+trc(\overline{G})=10=2n+2$
if both $G$ and $\overline{G}$ are connected.
Now, we give our main result.

\begin{thm}
Let $G$ be a graph of order~$n\geq 4$.
If both $G$ and $\overline{G}$ are connected,
then we have $trc(G)+trc(\overline{G})= 2n+2$ if $n=4$;
$trc(G)+trc(\overline{G})\leq 2n+1$ if $n=5$;
$trc(G)+trc(\overline{G})\leq 2n$ if $n\geq 6$.
Moreover, these upper bounds are tight.
\end{thm}
{\noindent\bf Proof.} The claim has already been established above for $n = 4$. We distinguish two cases according to the value of $n$.

$(i)$ $n=5$.
Firstly, we assume $G\cong P_5$
or $\overline{G}\cong P_5$.
Let $P_5=v_1v_2v_3v_4v_5$.
Now we define a total-coloring of $\overline{P_5}$
as follows: let $c(v_1v_4)=c(v_5v_2)=0,c(v_1v_3)=c(v_1v_5)=c(v_3v_5)=c(v_2v_4)=1$,
and assign the color~$2$
to each vertex. It can be easily checked
that this total-coloring with~$3$ colors makes $\overline{P_5}$
total-rainbow connected. It follows $trc(\overline{P_5})=3$.
Since $trc(P_5)=7$ by Theorem~\ref{thm3.9},
we obtain that $trc(G)+trc(\overline{G})=10<2n+1$ in this case.

Secondly, we suppose that $G$ or $\overline{G}\in\mathcal{T}^3$,
say $G\in \mathcal{T}^3$. There is only one possible element in $\mathcal{T}^3$,
whose complement is a graph in $\mathcal{H}_6$, implying $trc(\overline{G})=5$.
Since $trc(G)=6$ by Theorem~\ref{thm3.9},
we get $trc(G)+trc(\overline{G})=11$ in this case.
Thus, the upper bound is sharp when $n=5$.

In other cases, we have $trc(G)\leq 5$ and $trc(\overline{G})\leq 5$,
so $trc(G)+trc(\overline{G})\leq 10<2n+1$.
Therefore, our result holds when $n=5$.

$(ii)$ $n\geq 6$. We first deal with graphs that have
total-rainbow connection number at least $2n-6$, and then other graphs.

$(a)$ 
Suppose that $G\cong P_n$, that is, $trc(G)=2n-3$
by Theorem~\ref{thm3.9}.
Let $P_n=v_1v_2\dots v_n$.
Based on the total-coloring of $\overline{P_5}$ defined as above,
we set $c(v_1v_i)=i($mod $2)$, and $c(v_i)=2$
for each $6\leq i\leq n$, and color the remaining edges
arbitrarily (e.g., all of them are colored with~$0$).
Obviously, the vertices $v_j$ and $v_{j+1}$
are total-rainbow connected by the path $v_jv_1v_{j+1}$,
where $5\leq j\leq n-1$.
So this total-coloring with~$3$ colors makes $\overline{G}$
total-rainbow connected, and $trc(\overline{G})=3$ holds
since $trc(\overline{G})\geq 3$.
Thus, we have that $trc(G)+trc(\overline{G})=2n$.
The argument is the same when $\overline{G}\cong P_n$
and the result holds in this case.
Moreover, the upper bound is tight.

$(b)$ We now suppose that $G\in \mathcal{T}^3$, that is,
$trc(G)=2n-4$
by Theorem~\ref{thm3.9}.
Let $v$ be the vertex of $G$ with degree~$3$,
and let $x,y,z$ be the three pendent vertices of $G$.
Then set $P_{vx}=vu_1\dots u_k,
P_{vy}=vv_1\dots v_l,
P_{vz}=vw_1\dots w_m$,
where $x=u_k,y=v_l$ and $z=w_m$, respectively.
We have $k\geq l\geq m\geq 1$; moreover, $k\geq 2$
since $k+l+m+1=n\geq 6$.
We provide a total-coloring of $\overline{G}$ as follows:
let $c(v_1u_1)=c(vu_2)=c(u_1v_2)=c(u_1u_3)=0,
c(vv_2)=c(vu_3)=1$ (if $u_3$ or $v_2$ exists),
$c(u_2v_i)=i\pmod 2, 1\leq i\leq l$,
$c(u_2w_i)=i\pmod 2, 1\leq i\leq m$,
$ c(v_1u_i)=i+1\pmod 2, 3\leq i\leq k$, assign the color~$2$
to each vertex, and color the remaining edges
arbitrarily (e.g., all of them are colored with~$0$).
The vertex pair $(v,u_1)$ is total-rainbow connected by a path
$vu_3u_1$ or $vv_2u_1$ (note that at least one of the vertices $u_3$
and $v_2$ must exist),
and the vertex pair $(v,v_1)$ (or $(v,w_1)$) is total-rainbow
connected by a path
$vu_2v_1$ (or $vu_2w_1$).
The vertex pair $(u_i,u_{i+1})$ is total-rainbow connected by a path
$u_iv_1u_{i+1}$.
Similarly, we can find the total-rainbow paths connecting the vertex pairs
$(v_i,v_{i+1})$ and $(w_i,w_{i+1})$, respectively.
So this is a $TRC$-coloring of $\overline{G}$ with~$3$ colors,
implying $trc(\overline{G})\leq 3$.
Thus, we get $trc(G)+trc(\overline{G})\leq 2n-1$.
The argument is the same when $\overline{G}\in \mathcal{T}^3$
and the result holds in this case.

\begin{figure}[!t]
\centering
\scalebox{0.9}[0.9]{\includegraphics{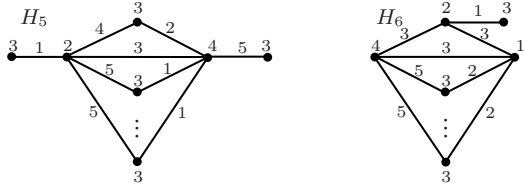}}
\caption{The graphs $H_5$ and $H_6$.}
\end{figure}

$(c)$  Then we turn to the case where
$2n-6\leq trc(G)\leq 2n-5$.
At first, we suppose
$G$ has at least two pendent edges, say $v_1v'_1$ and $v_2v_2'$,
where $v_1$ and $v_2$ are two pendent vertices of $G$.
If $v'_1\neq v'_2$, then $\overline{G}$ contains an $H_5$
as its spanning subgraph;
if $v'_1=v'_2$, then $\overline{G}$ contains an $H_6$
as its spanning subgraph.
It is easy to check that the total-colorings shown in Fig~$5$
make $H_5$ and $H_6$ total-rainbow connected, respectively.
It follows from Proposition~\ref{pro1} that $trc(\overline{G})\leq 5$, and
so $trc(G)+trc(\overline{G})\leq 2n$ holds.
In the following, we suppose $G$ has only one pendent vertex.

We assume that $G$ is a graph in $\mathcal{H}_6$ with one
pendent vertex $v$. We have $trc(G)=2n-5$ by Theorem~\ref{thm3.9}.
Let $P_{u_1v}=u_1u_5u_6\dots u_n$,
where $u_n=v$.
When $n=6$,
we define a total-coloring of $\overline{G}$:
let $c(u_1u_6)=c(u_1u_3)=c(u_2u_5)=c(u_3u_6)=0,
c(u_2u_6)=c(u_3u_5)=c(u_4u_5)=c(u_4u_6)=1$
and assign~$2$ to each vertex.
This is obviously a $TRC$-coloring of $\overline{G}$
(in fact, $\overline{G}$ is a tricyclic
graph with $\ell=5$, so we have $trc(\overline{G})=3$
from Theorem~\ref{thm3.6}).
Therefore, we get $trc(G)+trc(\overline{G})= 2n-2$.
When $n\geq 7$,
based on the total-coloring defined as above,
we set $c(u_3u_i)=i\pmod 2$, and $c(u_i)=2$ for each $7\leq i\leq n$.
Color the remaining edges
arbitrarily (e.g., all of them are colored with~$0$).
Obviously,
the vertex pair $(u_j,u_{j+1})$ is total-rainbow connected by
the path $u_ju_3u_{j+1}$, where $6\leq j\leq n-1$.
Thus, this total-coloring with~$3$ colors
makes $\overline{G}$ total-rainbow connected,
implying $trc(\overline{G})=3$.
Hence, we have $trc(G)+trc(\overline{G})= 2n-2<2n$ in this subcase.
Similarly, we have $trc(G)+trc(\overline{G})= 2n-2<2n$ if $G\in\{B_3,B_4\}$.

In total, we have proved $trc(G)+trc(\overline{G})\leq 2n$
if $G$ satisfies $2n-6\leq trc(G)\leq 2n-5$.
The argument is the same if $2n-6\leq trc(\overline{G})\leq 2n-5$.
Therefore, the result holds in this case.

$(d)$ In the following argument, we assume that
$3\leq trc(G)\leq 2n-7$ and $3\leq trc(\overline{G})\leq 2n-7$.
Obviously, the result holds when $n=6,7$.
Thus, we assume that $n\geq 8$ in the following.
By the connectivity we know that the diameters of $G$ and
$\overline{G}$ are both greater than~$1$.
So we consider the following four cases.
By symmetry, we suppose $diam(G)\geq diam(\overline{G})$.

$i)$ If  $diam(G)>3$, by Theorem~\ref{thmS}, we get $trc(\overline{G})\leq 7$, thus $trc(G)+trc(\overline{G})\leq 2n$.

$ii)$ If $diam(G)=diam(\overline{G})=3$, then by Theorem~\ref{thmZ},
both $G$ and $\overline{G}$
have a spanning subgraph which is a double star, say $T_1$ and $T_2$, respectively.
By Lemma~\ref{lem4.3}, we have $trc(G)\leq n+1$,
with equality if and only $G\cong T_1$.
Similarly we have $trc(\overline{G})\leq n+1$.
If one of $G$ and $\overline{G}$ is isomorphic to a double star,
say $G\cong T_1$, then $\overline{G}$ contains
an $H_5$ as its spanning subgraph, thus $trc(G)+trc(\overline{G})\leq n+6\leq 2n$.
Otherwise, both $trc(G)$ and $trc(\overline{G})$ are less than $n+1$,
also implying $trc(G)+trc(\overline{G})\leq 2n$.
Hence, the result holds in this subcase.

$iii)$ If $diam(G)=diam(\overline{G})=2$, then we can deduce that
$G$ (and $\overline{G}$) is  $2$-connected,
otherwise $G$ (respectively $\overline{G}$) has an isolated vertex.
By Corollary~\ref{cor4.2},
both $trc(G)\leq n-1$ and $trc(\overline{G})\leq n-1$,
thus $trc(G)+trc(\overline{G})< 2n$.

$iv)$ If $diam(G)=3$ and $diam(\overline{G})=2$,
we consider whether $G$ is~$2$-connected.
Recall that $n\geq 8$.
From Theorem~\ref{thm3.9},
we note that under this condition, $trc(G)\leq 2n-8$.
And there are only four possible graphs $G$
with $trc(G)= 2n-8$ (in fact, $G\in
{\mathcal{G}_2^5}'\cup\mathcal{G}_3^5\cup\mathcal{H}_{13}$).
All of them have~$8$ vertices.
So $trc(G)+trc(\overline{G})\leq 2n$ holds when the graph $G$
satisfies $trc(G)= 2n-8$.
In the following argument, we assume that $trc(G)\leq 2n-9$.
Again $\overline{G}$ is~$2$-connected,
so $trc(\overline{G})\leq n-1$ by Corollary~\ref{cor4.2}.
Thus, it suffices to consider this case under the assumption $n\geq 11$.

$Case~1$. The graph $G$ has cut vertices.
Let $v$ be a cut vertex of $G$, let $G_1,G_2,\dots,G_k$ be the components of $G-v$,
and let $n_i$ be the number of vertices in $G_i$ for $1\leq i\leq k$ with
$n_1\leq \dots\leq n_k$.  We consider the following two subcases.

$Subcase~1.1.$ There exists a cut vertex $v$ of $G$ such that $n-1-n_k\geq 2$.
Since $\Delta(G)\leq n-2$, we have $n_k\geq 2$. We know that $\overline{G}-v$
contains a spanning complete bipartite subgraph $K_{n-1-n_k,n_k}$.
Hence, it follows from Theorem~\ref{thmH} that $trc(\overline{G}-v)\leq 7$.
Given a strong $TRC$-coloring of $K_{n-1-n_k,n_k}$ based on Theorem~\ref{thmH},
if we assign a new color $p$ to an edge between $v$ and $\overline{G}-v$,
and color the other edges and vertices arbitrarily
(e.g., all of them are colored with~$p$),
the resulting total-coloring makes $\overline{G}$ total-rainbow connected.
Thus, we have $trc(\overline{G})\leq 8$.
Together with the assumption $trc(G)\leq 2n-9$,
we have $trc(G)+trc(\overline{G})< 2n$ in this subcase.

$Subcase~1.2.$ Every cut vertex $u$ of $G$ satisfies that $n-1-n_k=1$.
If $G$ has at least~$2$ pendent edges, then $\overline{G}$ contains an $H_5$
as its spanning subgraph, thus we get $trc(G)+trc(\overline{G})<2n$.
If $G$ has only one pendant edge $uw$,
where $w$ is the pendent vertex of $G$,
then $G-w$ is~$2$-connected.
Thus we have $trc(G-w)\leq n-1$ from Corollary~\ref{cor4.2}.
Given a $TRC$-coloring of $G-w$,
if we assign a new color $p$ to the vertices $w$ and $u$,
and $q(\neq p)$ to the edge $uw$, respectively,
then the resulting total-coloring makes $G$ total-rainbow connected.
Thus, we get $trc(G)+trc(\overline{G})\leq n-1+2+n-1\leq 2n$,
and the result holds in this subcase.

$Case~2$. The graph $G$ is~$2$-connected.
Let $v$ be a vertex of $G$ such that $ecc_G(v)=3$.
For convenience, we relabel $X=N_G(v)$, $Y=N_G^2(v)$,
and $Z=N_G^3(v)$.
And let $k=|X|$, $l=|Y|$ and $m=|Z|$.
Clearly, both $k\geq 2$ and $l\geq 2$ hold.

$Subcase~2.1.$ $m=1$ and say $z\in Z$.
If $G-z$ is~$2$-edge-connected,
then we have $trc(G-z)\leq n-1$ from Theorem~\ref{thm4.1}.
Similarly to Subcase~$1.2$,
we get $trc(G)+trc(\overline{G})\leq 2n$.

Otherwise, $G-z$ has bridges, and let $e$ be a bridge of $G-z$.
If $e$ is not a pendent edge of $G-z$,
then the graph $G-z-e$ contains two components, each of which
has at least~$2$ vertices.
The graph $G-z$ contains a spanning complete bipartite subgraph.
With a similar argument as Subcase~$1.1$,
we obtain $trc(G)+trc(\overline{G})< 2n$.
If $e=uw$ is a pendent edge of $G-z$,
then one of the vertices $u$ and $w$, say $u$, is a neighbor of $z$,
otherwise, $uw$ is a bridge of $G$, a contradiction.
That means $u$ is a vertex of degree~$2$ in $G$ and $u\in Y$.
Moreover, $w\in X$.
Now, we provide a total-coloring of $\overline{G}$ with~$7$ colors by
first letting $c(e)=1$ for each edge
$e\in E_{\overline{G}}[z,X\setminus \{w\}]\cup
E_{\overline{G}}[u,Y\setminus \{u\}]$;
$c(z)=2$, $c(zv)=3$, $c(v)=4$, $c(vu)=5$, $c(u)=6$,
and
$c(e)=7$ for each edge $e\in E_{\overline{G}}[u,X\setminus \{w\}]\cup
E_{\overline{G}}[v,Y\setminus \{u\}]$.
Then set $c(zw)=7$ and color the other edges and vertices arbitrarily
(e.g., all of them are colored with~$1$).
It's easy to check this total-coloring is a $TRC$-coloring of $\overline{G}$,
and so we have $trc(\overline{G})\leq 7$.
Together with the assumption
$trc(G)\leq 2n-9$, we get $trc(G)+trc(\overline{G})< 2n$.

$Subcase~2.2.$ $m\geq 2$. The following claim holds.

{\bf Claim:} $trc(\overline{G})\leq 11$.
\begin{proof} We divide $Y$ in $\overline{G}$
into three parts $Y_1$, $Y_2$ and $Y_3$,
where

$Y_1=\{u\in Y:$ there exists a vertex $w$ in $X$
such that $uw\in E(\overline{G})\}$,

$Y_2=\{u\in Y\setminus Y_1:$
there exists a vertex $w$ in $Z$
such that $uw\in E(\overline{G})\}$,

and $Y_3=Y\setminus (Y_1\cup Y_2)$.

Note that pick an arbitrarily vertex $y_3\in Y_3$ in $\overline{G}$,
for each vertex $x\in X$, there exists a vertex $y_1\in Y_1$
such that $y_1$ is a common neighbor of the vertices $y_3$ and $x$,
since $diam(\overline{G})=2$.
Moreover, $\overline{G}[X,Z]$ is a complete bipartite graph.
Given a strong $TRC$-coloring of $\overline{G}[X,Z]$
based on Theorem~\ref{thmH},
we provide a total-coloring of $\overline{G}$.
Set $c(e)=8$ for each edge $e\in E_{\overline{G}}[Y_1,X]\cup
E_{\overline{G}}[Y_2,Z]$,
$c(e)=9$ for each edge $e\in E_{\overline{G}}[v,Z]$,
$c(v)=10$, and
$c(e)=11$ for each edge $e\in E_{\overline{G}}[v,Y]$.
For any vertex $y_1\in Y_1$, choose a neighbor $x$ of $y_1$ in $X$.
Set $c(y_1)\in \{1,2,\dots,7\}\setminus\{c(x),c(xz),c(z)\}$,
where $z$ is a neighbor of $x$ in $Z$.
For each edge $e\in E_{\overline{G}}[y_1, Y_3]$,
we set $c(e)\in \{1,2,\dots,7\}\setminus\{c(x),c(xz),c(z),c(y_1)\}$.
And we color the other edges and vertices arbitrarily
(e.g., all of them are colored with~$1$).

Now, we verify the above total-coloring makes $\overline{G}$ total-rainbow connected.
We only need to consider the vertex pairs $(y,y')\in Y\times Y$.
If $(y,y')\in Y_1\times Y$,
then the path $yxzvy'$ is a desired total-rainbow path,
where $x\in N_{\overline{G}}(y)\cap X$
and $z\in N_{\overline{G}}(x)\cap Z$.
If $(y,y')\in Y_2\times (Y_2\cup Y_3)$,
then the path $yzvy'$ is a desired total-rainbow path,
where $z$ is a neighbor of $y$ in $Z$.
If $(y,y')\in Y_3\times Y_3$,
then the path $yy''xzvy'$ is a desired total-rainbow path,
where $y''\in N_{\overline{G}}(y)\cap Y_1$,
$x\in N_{\overline{G}}(y'')\cap X$
and $z\in N_{\overline{G}}(x)\cap Z$.
Thus, this total-coloring is a $TRC$-coloring of $\overline{G}$,
and so $trc(\overline{G})\leq 11$.
\end{proof}

If $n\geq 11$, we see that $G$ contains a spanning bicyclic
subgraph with $\ell\geq 6$, since $diam(G)=3$ and $G$ is~$2$-connected.
With an easy calculation
based on Observation~\ref{obs1}, Theorem~\ref{thm2.1} and Theorem~\ref{thm2.3},
we get $trc(G)\leq 2n-11$ if $n\geq 11$.
Therefore, we obtain $trc(G)+trc(\overline{G})\leq 2n$ in this subcase.

Our proof is complete now.
\hspace{9.6cm}{$\blacksquare$}

\end{document}